\documentclass[12pt]{amsart}
\usepackage{amsmath,amsthm,amsfonts,amssymb,mathrsfs}
\date{\today}

\usepackage{color}
\input xymatrix
\xyoption{all}

\usepackage{hyperref}

\newtheorem{theorem}{Theorem}

\newtheorem{corollary}{Corollary}
\newtheorem{lemma}{Lemma}
\theoremstyle{definition}
\newtheorem{example}{Example}

\begin{document}

\title{On the group of automorphisms of the Brandt $\lambda^0$-extension of a monoid with zero}

\author{Oleg~Gutik}
\address{Faculty of Mathematics, National University of Lviv,
Universytetska 1, Lviv, 79000, Ukraine}
\email{o\underline{\hskip5pt}\,gutik@franko.lviv.ua,
ovgutik@yahoo.com}



\maketitle              

\begin{abstract}
The group of automorphisms of the Brandt $\lambda^0$-extension $B^0_\lambda(S)$ of an arbitrary monoid $S$ with zero is described.
In particular we show that the group of automorphisms $\mathbf{Aut}(B_{\lambda}^0(S))$ of $B_{\lambda}^0(S)$ is isomorphic to a homomorphic image of the group defines on the Cartesian product $\mathscr{S}_{\lambda}\times \mathbf{Aut}(S)\times H_1^{\lambda}$ with the following binary operation:
\begin{equation*}
    [\varphi,h,u]\cdot[\varphi^{\prime},h^{\prime},u^{\prime}]= [\varphi\varphi^{\prime},hh^{\prime},\varphi u^{\prime}\cdot uh^{\prime}],
\end{equation*}
where $\mathscr{S}_{\lambda}$ is the group of all bijections of the cardinal $\lambda$, $\mathbf{Aut}(S)$ is the group of all automorphisms of the semigroup $S$ and $H_1^{\lambda}$ is the direct $\lambda$-power of the group of units $H_1$ of the monoid $S$.
\end{abstract}

\section{Introduction and preliminaries}

Further we shall follow the terminology of \cite{Clifford-Preston-1961-1967, Petrich-1984}.

Given a semigroup $S$, we shall
denote the set of idempotents of $S$ by $E(S)$.
A
semigroup $S$ with the adjoined unit
(identity) [zero] will be denoted by
$S^1$ [$S^0$] (cf.
\cite{Clifford-Preston-1961-1967}).
Next, we shall denote the unit
(identity) and the zero of a semigroup $S$ by $1_S$ and $0_S$,
respectively. Given a subset $A$ of a semigroup $S$, we shall denote
by $A^*=A\setminus\{ 0_S\}$.

If $S$ is a semigroup, then we shall denote the subset of idempotents in $S$ by $E(S)$. If $E(S)$ is closed under multiplication in $S$ and we shall refer to $E(S)$ a \emph{band} (or the \emph{band of} $S$). If the band $E(S)$ is a non-empty subset of $S$, then the semigroup operation on $S$ determines the following partial order $\leqslant$ on $E(S)$: $e\leqslant f$ if and only if $ef=fe=e$. This order is called the {\em natural partial order} on $E(S)$.

If $h\colon S\rightarrow T$ is a homomorphism (or a map) from a semigroup $S$ into a semigroup $T$ and if $s\in S$, then we denote the image of $s$ under $h$ by $(s)h$.

Let $S$ be a semigroup with zero and $\lambda$ a cardinal $\geqslant 1$. We define the semigroup operation on the set $B_{\lambda}(S)=(\lambda\times S\times {\lambda})\cup\{ 0\}$ as follows:
\begin{equation*}
 (\alpha,a,\beta)\cdot(\gamma, b, \delta)=
  \begin{cases}
    (\alpha, ab, \delta), & \text{ if } \beta=\gamma; \\
    0, & \text{ if } \beta\ne \gamma,
  \end{cases}
\end{equation*}
and $(\alpha, a, \beta)\cdot 0=0\cdot(\alpha, a, \beta)=0\cdot
0=0,$ for all $\alpha, \beta, \gamma, \delta\in {\lambda}$ and $a,
b\in S$. If $S=S^1$ then the semigroup $B_\lambda(S)$ is called
the {\it Brandt $\lambda$-extension of the semigroup}
$S$~\cite{Gutik-1999}. Obviously, if $S$ has zero
then ${\mathcal J}=\{ 0\}\cup\{(\alpha, 0_S, \beta)\colon 0_S$ is
the zero of $S\}$ is an ideal of $B_\lambda(S)$. We put
$B^0_\lambda(S)=B_\lambda(S)/{\mathcal J}$ and the semigroup
$B^0_\lambda(S)$ is called the {\it Brandt $\lambda^0$-extension
of the semigroup $S$ with zero}~\cite{Gutik-Pavlyk-2006}.

If $\mathcal{I}$ is a trivial semigroup (i.e. $\mathcal{I}$ contains only one element), then we denote the semigroup $\mathcal{I}$ with the adjoined zero by ${\mathcal{I}}^0$. Obviously, for any $\lambda\geqslant 2$, the Brandt $\lambda^0$-extension of the semigroup ${\mathcal{I}}^0$ is isomorphic to the semigroup of $\lambda{\times}\lambda$-matrix units and any Brandt $\lambda^0$-extension of a semigroup with zero which also contains a non-zero idempotent contains the semigroup of $\lambda{\times}\lambda$-matrix units.
We shall denote  the semigroup of $\lambda{\times}\lambda$-matrix units by $B_\lambda$. The $2\times 2$-matrix semigroup with adjoined identity $B_2^1$ plays an impotent role in Graph Theory and its called the \emph{Perkins semigroup}. In the paper \cite{Perkins-1969} Perkins showed that the semigroup $B_2^1$ is not finitely based. More details on the word problem of the Perkins semigroup via different graphs may be found in the works of Kitaev and his coauthors (see \cite{Kitaev-Lozin=2015, Kitaev-Seif=2008}).

We always consider the Brandt
$\lambda^0$-extension only of a monoid with zero. Obviously, for
any monoid $S$ with zero we have $B^0_1(S)=S$. Note that every
Brandt $\lambda$-extension of a group $G$ is isomorphic to the
Brandt $\lambda^0$-extension of the group $G^0$ with adjoined
zero. The Brandt $\lambda^0$-extension of the group with adjoined
zero is called a \emph{Brandt semigroup}~\cite{Clifford-Preston-1961-1967, Petrich-1984}. A semigroup $S$ is a Brandt semigroup if and only if $S$ is a
completely $0$-simple inverse semigroup~\cite{Clifford-1942,
Munn-1957} (cf.  also \cite[Theorem~II.3.5]{Petrich-1984}).
We shall say that the Brandt
$\lambda^0$-extension $B_\lambda^0(S)$ of a semigroup $S$ is
\emph{finite} if the cardinal $\lambda$ is finite.

In the paper \cite{Gutik-Repovs-2010} Gutik and Repov\v{s} established homomorphisms of the Brandt $\lambda^0$-extensions of monoids with zeros. They also described a category whose objects are ingredients in the constructions of the Brandt $\lambda^0$-extensions of monoids with zeros. Here they introduced finite, compact topological Brandt $\lambda^0$-extensions of topological semigroups and countably compact topological Brandt $\lambda^0$-extensions of topological inverse semigroups in the class of topological inverse semigroups, and established the structure of such extensions and non-trivial continuous homomorphisms between such topological Brandt $\lambda^0$-extensions of topological  monoids with zero. There they also described a category whose objects are ingredients in the constructions of finite (compact,
countably compact) topological Brandt $\lambda^0$-extensions of topological  monoids with zeros. These  investigations were continued in \cite{Gutik-Pavlyk-Reiter-2009} and \cite{Gutik-Pavlyk-2013a}, where established countably compact topological Brandt $\lambda^0$-extensions of topological monoids with zeros and pseudocompact topological Brandt $\lambda^0$-extensions of semitopological monoids with zeros their corresponding categories. Some other topological aspects of topologizations, embeddings and completions of the semigroup of $\lambda{\times}\lambda$-matrix units and Brandt $\lambda^0$-extensions as semitopological and topological semigroups were studied in \cite{Bardyla-Gutik-2016, Gutik-2014, Gutik-Pavlyk-2005, Gutik-Pavlyk-Reiter-2011, Gutik-Ravsky-2015, Gutik-Ravsky-2015a, Jamalzadeh-Rezaei-2010, Jamalzadeh-Rezaei-2013}.

In this paper we describe the group of automorphisms of the Brandt $\lambda^0$-extension $B^0_\lambda(S)$ of an arbitrary monoid $S$ with zero.

\section{Automorphisms of the Brandt $\lambda^0$-extension of a monoid with zero}

We observe that if $f\colon S\to S$ is an automorphism of the semigroup $S$ without zero then it is obvious that the map $\widehat{f}\colon S^0\to S^0$ defined by the formula
\begin{equation*}
    (s)\widehat{f}=
\left\{
  \begin{array}{cl}
    (s)f, & \hbox{if~} s\neq 0_S;\\
    0_S, & \hbox{if~} s=0_S,
  \end{array}
\right.
\end{equation*}
is an automorphism of the semigroup $S^0$ with adjoined zero $0_S$. Also the automorphism $f\colon S\to S$ of the semigroup $S$ can be extended to an automorphism $f_B\colon B_\lambda^0(S)\to B_\lambda^0(S)$ of the Brandt $\lambda^0$-extension $B_\lambda^0(S)$ of the semigroup $S$ by the formulae:
\begin{equation*}
    (\alpha,s,\beta)f_B=(\alpha,(s)f,\beta), \quad \hbox{~for all~~} \alpha,\beta\in\lambda
\end{equation*}
and $(0)f_B=0$. We remark that so determined extended automorphism is not unique.

The following theorem describes all automorphisms of the Brandt $\lambda^0$-extension $B_\lambda^0(S)$ of a monoid $S$.

\begin{theorem}\label{theorem-3.1}
Let $\lambda\geqslant 1$ be cardinal and let $B_{\lambda}^0(S)$ be the Brandt $\lambda^0$-extension of monoid $S$ with zero. Let $h\colon S\rightarrow S$ be an automorphism and suppose that $\varphi\colon \lambda\rightarrow \lambda$ is a bijective map. Let $H_1$ be the group of units of $S$ and $u\colon \lambda\rightarrow H_1$ a map. Then the map $\sigma\colon B_{\lambda}^0(S)\rightarrow
B_{\lambda}^0(S)$ defined by the formulae
\begin{equation}\label{eq-3.1}
  ((\alpha,s,\beta))\sigma=  ((\alpha)\varphi,(\alpha)u\cdot(s)h\cdot((\beta)u)^{-1},(\beta)\varphi) \qquad  \hbox{~and~} \qquad (0)\sigma=0,
\end{equation}
is an automorphism of the semigroup $B_{\lambda}^0(S)$. Moreover, every automorphism of $B_{\lambda}^0(S)$ can be constructed in this manner.
\end{theorem}

\begin{proof}
A simple verification shows that $\sigma$ is an automorphism of the semigroup
$B_{\lambda}^0(S)$.

Let $\sigma\colon B_{\lambda}^0(S)\rightarrow B_{\lambda}^0(S)$ be an isomorphism. We fix an arbitrary $\alpha\in \lambda$.

Since $\sigma\colon
B_{\lambda}^0(S)\rightarrow B_{\lambda}^0(S)$ is the automorphism and the idempotent $(\alpha,1_S,\alpha)$ is maximal with the respect to the natural partial order on $E(B_{\lambda}^0(S))$, Proposition~3.2 of \cite{Gutik-Repovs-2010} implies that $((\alpha,1_S,\alpha))\sigma=(\alpha^{\,\prime},1_S,\alpha^{\,\prime})$ for some $\alpha^{\,\prime}\in \lambda$.

Since $(\beta,1_S,\alpha)(\alpha,1_S,\alpha)=(\beta,1_S,\alpha)$ for any $\beta\in \lambda$, we have that
\begin{equation*}
  ((\beta,1_S,\alpha))\sigma=
  ((\beta,1_S,\alpha))\sigma\cdot
  (\alpha^{\,\prime},1_S,\alpha^{\,\prime}),
\end{equation*}
and hence
\begin{equation*}
  ((\beta,1_S,\alpha))\sigma=
  ((\beta)\varphi,(\beta)u,\alpha^{\,\prime}),
\end{equation*}
for some $(\beta)\varphi\in\lambda$ and $(\beta)u\in S$.
Similarly, we get that
\begin{equation*}
  ((\alpha,1_S,\beta))\sigma=
  (\alpha^{\,\prime},(\beta)v,(\beta)\psi),
\end{equation*}
for some $(\beta)\psi\in\lambda$ and $(\beta)v\in S$. Since
$(\alpha,1_S,\beta)(\beta,1_S,\alpha)=(\alpha,1_S,\alpha)$, we
have that
\begin{equation*}
  (\alpha^{\,\prime},1_S,\alpha^{\,\prime}) = ((\alpha,1_S,\alpha))\sigma= (\alpha^{\,\prime},(\beta)v,(\beta)\psi)\cdot
    ((\beta)\varphi,(\beta)u,\alpha^{\,\prime})= (\alpha^{\,\prime},(\beta)v\cdot(\beta)u,\alpha^{\,\prime}),
\end{equation*}
and hence $(\beta)\varphi=(\beta)\psi=\beta^{\,\prime}\in \lambda$ and $(\beta)v\cdot(\beta)u=1_S$. Similarly, since
$(\beta,1_S,\alpha)\cdot(\alpha,1_S,\beta)=(\beta,1_S,\beta)$, we
see that the element
\begin{equation*}
 ((\beta,1_S,\beta))\sigma = ((\beta,1_S,\alpha)(\alpha,1_S,\beta))\sigma=
    (\beta^{\,\prime},(\beta)v\cdot(\beta)u,\beta^{\,\prime})
\end{equation*}
is a maximal idempotent of the subsemigroup $S_{\beta^{\,\prime},\beta^{\,\prime}}$ of $B_{\lambda}^0(S)$, and hence we have that $(\beta)v\cdot(\beta)u=1_S$. This implies that the elements $(\beta)v$ and $(\beta)u$ are mutually invertible in $H_1$, and hence $(\beta)v=((\beta)u)^{-1}$.

  If
$(\gamma)\varphi=(\delta)\varphi$ for $\gamma,\delta\in \lambda$ then
\begin{equation*}
  0\neq  (\alpha^{\,\prime},1_S,(\gamma)\varphi)\cdot     ((\delta)\varphi,1_S,\alpha^{\,\prime})=
  ((\alpha,1_S,\gamma))\sigma\cdot((\delta,1_S,\alpha))\sigma,
\end{equation*}
and since $\sigma$ is an automorphism, we have that
\begin{equation*}
(\alpha,1_S,\gamma)\cdot(\delta,1_S,\alpha)\neq 0
\end{equation*}
and hence $\gamma=\delta$. Thus $\varphi\colon \lambda\rightarrow\lambda$ is a bijective map.

Therefore for $s\in S\setminus\{0_S\}$ we have
\begin{equation*}
\begin{split}
((\gamma,s,\delta))\sigma=&\,
    ((\gamma,1_S,\alpha)\cdot(\alpha,s,\alpha)\cdot(\alpha,1_S,\delta))\sigma=\\
    =&\,((\gamma,1_S,\alpha))\sigma\cdot ((\alpha,s,\alpha))\sigma
    \cdot((\alpha,1_S,\delta))\sigma=\\
    =&\,((\gamma)\varphi,(\gamma)u,\alpha^{\,\prime})
    \cdot(\alpha^{\,\prime},(s)h,\alpha^{\,\prime})
    \cdot(\alpha^{\,\prime},((\delta)u)^{-1},(\delta)\varphi){=}\\
    =&\,((\gamma)\varphi,(\gamma)u\cdot(s)h\cdot((\delta)u)^{-1},(\delta)\varphi).
\end{split}
\end{equation*}
Also, since $0$ is zero of the semigroup $B_{\lambda}^0(S)$ we conclude that
$(0)\sigma=0$.
\end{proof}

Theorem~\ref{theorem-3.1} implies the following corollary:

\begin{corollary}\label{corollary-3.2}
Let $\lambda\geqslant 1$ be cardinal and let $B_{\lambda}(G)$ be the Brandt semigroup. Let $h\colon G\rightarrow G$ be an automorphism and suppose that $\varphi\colon \lambda\rightarrow \lambda$ is a bijective map. Let $u\colon \lambda\rightarrow G$ be a map. Then the map $\sigma\colon B_{\lambda}(G)\rightarrow
B_{\lambda}(G)$ defined by the formulae
\begin{equation*}
  ((\alpha,s,\beta))\sigma=  ((\alpha)\varphi,(\alpha)u\cdot(s)h\cdot((\beta)u)^{-1},(\beta)\varphi) \qquad \hbox{~and~} \qquad (0)\sigma=0,
\end{equation*}
is an automorphism of the Brandt semigroup $B_{\lambda}(G)$. Moreover, every automorphism of $B_{\lambda}(G)$ can be constructed in this manner.
\end{corollary}

Also, we observe that Corollary~\ref{corollary-3.2} implies the following well known statement:

\begin{corollary}\label{corollary-3.3}
Let $\lambda\geqslant 1$ be cardinal and $\varphi\colon \lambda\rightarrow \lambda$ a bijective map. Then the map $\sigma\colon B_{\lambda}\rightarrow
B_{\lambda}$ defined by the formulae
\begin{equation*}
    ((\alpha,\beta))\sigma=((\alpha)\varphi,(\beta)\varphi) \quad \hbox{and} \quad (0)\sigma=0,
\end{equation*}
is an automorphism of the semigroup of $\lambda{\times}\lambda$-matrix units $B_{\lambda}$. Moreover, every automorphism of $B_{\lambda}$ can be constructed in this manner.
\end{corollary}

The following example implies that the condition that semigroup $S$ contains the identity is essential.

\begin{example}\label{example-3.4}
Let $\lambda$ be any cardinal $\geqslant 2$. Let $S$ be the zero-semigroup of cardinality $\geqslant 3$ and $0_S$ is zero of $S$. It is easily to see that every bijective map $\sigma\colon B_{\lambda}^0(S)\rightarrow B_{\lambda}^0(S)$ such that $(0)\sigma=0$ is an automorphism of the Brandt $\lambda^0$-extension of $S$.
\end{example}

\noindent
\textbf{Remark.}
By Theorem~\ref{theorem-3.1} we have that every automorphism $\sigma\colon B_{\lambda}^0(S)\rightarrow B_{\lambda}^0(S)$ of the Brandt $\lambda^0$-extension of an arbitrary monoid $S$ with zero identifies with the ordered triple $[\varphi,h,u]$, where $h\colon S\rightarrow S$ is an automorphism of $S$, $\varphi\colon \lambda\rightarrow \lambda$ is a bijective map and $u\colon \lambda\rightarrow H_1$ is a map, where $H_1$ is the group of units of $S$.

\begin{lemma}\label{lemma-3.6}
Let $\lambda\geqslant 1$ be cardinal, $S$ be a monoid with zero and let $B_{\lambda}^0(S)$ be the Brandt $\lambda^0$-extension of $S$. Then the composition of arbitrary automorphisms $\sigma=[\varphi,h,u]$ and $\sigma^{\prime}=[\varphi^{\prime},h^{\prime},u^{\prime}]$ of the Brandt $\lambda^0$-extension of $S$ defines in the following way:
\begin{equation*}
    [\varphi,h,u]\cdot[\varphi^{\prime},h^{\prime},u^{\prime}]= [\varphi\varphi^{\prime},hh^{\prime},\varphi u^{\prime}\cdot uh^{\prime}].
\end{equation*}
\end{lemma}

\begin{proof}
By Theorem~\ref{theorem-3.1} for every $(\alpha,s,\beta)\in B_{\lambda}^0(S)$ we have that
\begin{equation*}
\begin{split}
(\alpha,s,\beta)(\sigma\sigma^{\prime})& =  \left((\alpha)\varphi,(\alpha)u{\cdot}(s)h{\cdot}((\beta)u)^{-1},(\beta)\varphi\right) \sigma^{\prime}{=}\\
    & = \big(((\alpha)\varphi)\varphi^{\prime}, ((\alpha)\varphi)u^{\prime}\cdot \left((\alpha)u\cdot(s)h\cdot((\beta)u)^{-1}\right)h^{\prime} \cdot  \left(((\beta)\varphi)u^{\prime}\right)^{-1}, ((\beta)\varphi)\varphi^{\prime}\big)=
\end{split}
\end{equation*}
and since $h^{\prime}$ is an automorphism of the monoid $S$ we get that this is equal to
\begin{equation*}
\begin{split}
  =& \big(((\alpha)\varphi)\varphi^{\prime}, ((\alpha)\varphi)u^{\prime}\cdot
  \left((\alpha)u\right)h^{\prime}\cdot
  \left((s)h\right)h^{\prime} \cdot\left(((\beta)u)h^{\prime}\right)^{-1}\cdot \left(((\beta)\varphi)u^{\prime}\right)^{-1}, ((\beta)\varphi)\varphi^{\prime}\big)=\\
  =& \big((\alpha)(\varphi\varphi^{\prime}), (\alpha)(\varphi u^{\prime}\cdot
  uh^{\prime})\cdot
  \left((s)h\right)h^{\prime} \cdot (\beta)\left(\varphi u^{\prime}\cdot
  uh^{\prime}\right)^{-1}, (\beta)(\varphi\varphi^{\prime})\big).
\end{split}
\end{equation*}
This completes the proof of the requested equality.
\end{proof}

\begin{theorem}\label{theorem-3.7}
Let $\lambda\geqslant 1$ be cardinal, $S$ be a monoid with zero and let $B_{\lambda}^0(S)$ be the Brandt $\lambda^0$-extension of $S$. Then the group of automorphisms $\mathbf{Aut}(B_{\lambda}^0(S))$ of $B_{\lambda}^0(S)$ is isomorphic to a homomorphic image of the group defines on the Cartesian product $\mathscr{S}_{\lambda}\times \mathbf{Aut}(S)\times H_1^{\lambda}$ with the following binary operation:
\begin{equation}\label{eq-3.2}
    [\varphi,h,u]\cdot[\varphi^{\prime},h^{\prime},u^{\prime}]= [\varphi\varphi^{\prime},hh^{\prime},\varphi u^{\prime}\cdot uh^{\prime}],
\end{equation}
where $\mathscr{S}_{\lambda}$ is the group of all bijections of the cardinal $\lambda$, $\mathbf{Aut}(S)$ is the group of all automorphisms of the semigroup $S$ and $H_1^{\lambda}$ is the direct $\lambda$-power of the group of units $H_1$ of the monoid $S$. Moreover, the inverse element of $[\varphi,h,u]$ in the group $\mathbf{Aut}(B_{\lambda}^0(S))$ is defined by the formula:
\begin{equation*}
    [\varphi,h,u]^{-1}=\left[\varphi^{-1},h^{-1},\varphi^{-1}u^{-1}h^{-1}\right].
\end{equation*}
\end{theorem}

\begin{proof}
First, we show that the binary operation defined by formula (\ref{eq-3.2}) is associative. Let $[\varphi,h,u]$, $[\varphi^{\prime},h^{\prime},u^{\prime}]$ and $[\varphi^{\prime\prime},h^{\prime\prime},u^{\prime\prime}]$ be arbitrary elements of the Cartesian product $\mathscr{S}_{\lambda}\times \mathbf{Aut}(S)\times H_1^{\lambda}$. Then we have that
\begin{equation*}
\begin{split}
\big([\varphi,h,u]\cdot[\varphi^{\prime},h^{\prime},u^{\prime}]\big) \cdot [\varphi^{\prime\prime},h^{\prime\prime},u^{\prime\prime}]& =[\varphi\varphi^{\prime},hh^{\prime},\varphi u^{\prime}\cdot uh^{\prime}]\cdot [\varphi^{\prime\prime},h^{\prime\prime},u^{\prime\prime}]=\\
&= [\varphi\varphi^{\prime}\varphi^{\prime\prime},hh^{\prime}h^{\prime\prime}, \varphi\varphi^{\prime}u^{\prime\prime}\cdot (\varphi u^{\prime}\cdot uh^{\prime})h^{\prime\prime}]=\\
&= [\varphi\varphi^{\prime}\varphi^{\prime\prime},hh^{\prime}h^{\prime\prime}, \varphi\varphi^{\prime}u^{\prime\prime}\cdot \varphi u^{\prime}h^{\prime\prime}\cdot uh^{\prime}h^{\prime\prime}]
\end{split}
\end{equation*}
and
\begin{equation*}
\begin{split}
[\varphi,h,u]\cdot\left([\varphi^{\prime},h^{\prime},u^{\prime}] \cdot [\varphi^{\prime\prime},h^{\prime\prime},u^{\prime\prime}]\right)&=
 [\varphi,h,u]\cdot [\varphi^{\prime}\varphi^{\prime\prime},h^{\prime}h^{\prime\prime}, \varphi^{\prime} u^{\prime\prime}\cdot u^{\prime}h^{\prime\prime}]=\\
&= [\varphi\varphi^{\prime}\varphi^{\prime\prime},hh^{\prime}h^{\prime\prime}, \varphi(\varphi^{\prime}u^{\prime\prime}\cdot u^{\prime}h^{\prime\prime})\cdot uh^{\prime}h^{\prime\prime}]=\\
&= [\varphi\varphi^{\prime}\varphi^{\prime\prime},hh^{\prime}h^{\prime\prime}, \varphi\varphi^{\prime}u^{\prime\prime}\cdot \varphi u^{\prime}h^{\prime\prime}\cdot uh^{\prime}h^{\prime\prime}],
\end{split}
\end{equation*}
and hence so defined operation is associative.

Theorem~\ref{theorem-3.1} implies that formula (\ref{eq-3.1}) determines a map $\mathfrak{F}$ from the Cartesian product $\mathscr{S}_{\lambda}\times \mathbf{Aut}(S)\times H_1^{\lambda}$ onto the group of automorphisms $\mathbf{Aut}(B_{\lambda}^0(S))$ of the Brandt $\lambda^0$-extension $B_{\lambda}^0(S)$ of the monoid $S$, and hence the associativity of binary operation (\ref{eq-3.2}) implies that the map $\mathfrak{F}$ is a homomorphism from $\mathscr{S}_{\lambda}\times \mathbf{Aut}(S)\times H_1^{\lambda}$ onto the group $\mathbf{Aut}(B_{\lambda}^0(S))$.

Next we show that $[1_{\mathscr{S}_{\lambda}},1_{\mathbf{Aut}(S)},1_{H_1^{\lambda}}]$ is a unit element with the respect to the binary operation (\ref{eq-3.2}), where $1_{\mathscr{S}_{\lambda}}$, $1_{\mathbf{Aut}(S)}$ and $1_{H_1^{\lambda}}$ are units of the groups $\mathscr{S}_{\lambda}$, $\mathbf{Aut}(S)$ and $H_1^{\lambda}$, respectively. Then we have that
\begin{equation*}
\begin{split}
  [\varphi,h,u]\cdot \big[1_{\mathscr{S}_{\lambda}},1_{\mathbf{Aut}(S)},1_{H_1^{\lambda}}\big] &
     = \big[\varphi1_{\mathscr{S}_{\lambda}},h1_{\mathbf{Aut}(S)}, \varphi1_{H_1^{\lambda}}\cdot u1_{\mathbf{Aut}(S)}\big]=\\
    & = \big[\varphi,h,\varphi 1_{H_1^{\lambda}}\cdot u1_{\mathbf{Aut}(S)}\big]= \\
    & = \big[\varphi,h, 1_{H_1^{\lambda}}\cdot u\big]= \\
    & = [\varphi,h, u]
\end{split}
\end{equation*}
and
\begin{equation*}
  \big[1_{\mathscr{S}_{\lambda}},1_{\mathbf{Aut}(S)},1_{H_1^{\lambda}}\big]  \cdot[\varphi,h,u]=
  \big[1_{\mathscr{S}_{\lambda}}\varphi,1_{\mathbf{Aut}(S)}h, 1_{\mathscr{S}_{\lambda}}u\cdot 1_{H_1^{\lambda}}h\big]=
  [\varphi,h, u],
\end{equation*}
because every automorphism $h\in\mathbf{Aut}(S)$ acts on the group $H_1^{\lambda}$ by the natural way as a restriction of global automorphism of the semigroup $S$ on every factor, and hence we get that $1_{H_1^{\lambda}}h=1_{H_1^{\lambda}}$.

Also, similar arguments imply that
\begin{equation*}
\begin{split}
  [\varphi,h,u]\cdot[\varphi,h, u]^{-1} & =
  [\varphi,h,u]\cdot \left[\varphi^{-1},h^{-1},\varphi^{-1}u^{-1}h^{-1}\right]{=}\\
    & =
    \left[\varphi\varphi^{-1},hh^{-1},(\varphi\varphi^{-1})u^{-1}h^{-1}\cdot uh^{-1}\right]=\\
    & =
    \left[\varphi\varphi^{-1},hh^{-1},(1_{\mathscr{S}_{\lambda}})u^{-1}h^{-1}\cdot uh^{-1}\right]=\\
    & =
    \left[\varphi\varphi^{-1},hh^{-1},u^{-1}h^{-1}\cdot uh^{-1}\right]=\\
    & = \big[1_{\mathscr{S}_{\lambda}},1_{\mathbf{Aut}(S)},1_{H_1^{\lambda}}\big]
\end{split}
\end{equation*}
and
\begin{equation*}
\begin{split}
[\varphi,h, u]^{-1}\cdot [\varphi,h,u] &=
  \left[\varphi^{-1},h^{-1},\varphi^{-1}u^{-1}h^{-1}\right]\cdot[\varphi,h,u] {=}\\
   = & \left[\varphi^{-1}\varphi,h^{-1}h,\varphi^{-1}u\cdot \varphi^{-1}u^{-1}h^{-1}h\right]=\\
   = & \left[\varphi^{-1}\varphi,h^{-1}h,\varphi^{-1}u\cdot \varphi^{-1}u^{-1}\right]=\\
   = & \; \big[1_{\mathscr{S}_{\lambda}},1_{\mathbf{Aut}(S)},1_{H_1^{\lambda}}\big].
\end{split}
\end{equation*}
This implies that the elements $\left[\varphi^{-1},h^{-1},\varphi^{-1}u^{-1}h^{-1}\right]$ and $[\varphi,h,u]$ are invertible in $\mathscr{S}_{\lambda}\times \mathbf{Aut}(S)\times H_1^{\lambda}$, and hence the set $\mathscr{S}_{\lambda}\times \mathbf{Aut}(S)\times H_1^{\lambda}$ with the binary operation (\ref{eq-3.2}) is a group.

Let $\mathbf{Id}\colon B_{\lambda}^0(S)\rightarrow B_{\lambda}^0(S)$ be the identity automorphism of the semigroup $B_{\lambda}^0(S)$. Then by Theorem~\ref{theorem-3.1} there exist some automorphism $h\colon S\rightarrow S$, a bijective map  $\varphi\colon \lambda\rightarrow \lambda$ and a map $u\colon \lambda\rightarrow H_1$ into the group $H_1$ of units of $S$ such that
\begin{equation*}
  (\alpha,s,\beta)=(\alpha,s,\beta)\mathbf{Id}= ((\alpha)\varphi,(\alpha)u\cdot(s)h\cdot((\beta)u)^{-1},(\beta)\varphi),
\end{equation*}
for all $\alpha,\beta\in\lambda$ and $s\in S^*$. Since $\mathbf{Id}\colon B_{\lambda}^0(S)\rightarrow B_{\lambda}^0(S)$ is the identity automorphism we conclude that $(\alpha)\varphi=\alpha$ for every $\alpha\in\lambda$. Also, for every $s\in S^*$ we get that $s=(\alpha)u\cdot(s)h\cdot((\beta)u)^{-1}$ for all $\alpha,\beta\in\lambda$, and hence we obtain that
\begin{equation*}
1_S=(\alpha)u\cdot(1_S)h\cdot((\beta)u)^{-1}= (\alpha)u\cdot((\beta)u)^{-1}
\end{equation*}
for all $\alpha,\beta\in\lambda$. This implies that $(\alpha)u=(\beta)u=\widetilde{u}$ is a fixed element of the group $H_1$ for all $\alpha,\beta\in\lambda$.

We define
\begin{equation*}
  \ker N=  \; \Big\{[\varphi,h,\widetilde{u}]\in \mathscr{S}_{\lambda}\times \mathbf{Aut}(S)\times H_1^{\lambda}\colon
     \varphi\colon\lambda\rightarrow\lambda \hbox{~~is an idemtity map,~}\\
     \widetilde{u}(s)h\widetilde{u}\;^{-1}=s \hbox{~~for any~~} s\in S\Big\}.
\end{equation*}
It is obvious that the equality $\widetilde{u}(s)h\widetilde{u}^{-1}=s$ implies that $(s)h=\widetilde{u}^{-1}s\widetilde{u}$ for all $s\in S$. The previous arguments implies that $[\varphi,h,\widetilde{u}]\in\ker N$ if and only if $[\varphi,h,\widetilde{u}]\mathfrak{F}$ is the unit of the group $\mathbf{Aut}(B_{\lambda}^0(S))$, and hence $\ker N$ is a normal subgroup of $\mathscr{S}_{\lambda}\times \mathbf{Aut}(S)\times H_1^{\lambda}$. This implies that the quotient group $(\mathscr{S}_{\lambda}\times \mathbf{Aut}(S)\times H_1^{\lambda})/\ker N$ is isomorphic to the group $\mathbf{Aut}(B_{\lambda}^0(S))$.
\end{proof}

\end{document}